\documentclass[10pt,reqno]{amsart}
\usepackage{fullpage,amssymb,amsmath,mathrsfs,enumerate,amsthm,amscd,amsfonts}

\newtheorem{thm}{Theorem}[section]
\newtheorem{proposition}[thm]{Proposition}
\newtheorem{lemma}[thm]{Lemma}
\newtheorem{cor}[thm]{Corollary}
\newtheorem{exa}[thm]{Example}

\theoremstyle{definition}
\newtheorem{definition}[thm]{Definition}
\newtheorem{remark}[thm]{Remark}
\newtheorem{question}{Question}
\newtheorem{problem}{Problem}
\newtheorem{convention}[thm]{Convention}
\newtheorem*{ack}{Acknowledgments}

\def\conv{\operatorname{conv}}
\def\er{\mathbb R}
\def\R{\mathcal R}
\def\C{\mathcal{C}}
\def\RC{\mathcal R \mathcal C}
\def\GO{\mathcal G \Omega}
\def\tn{|\!|\!|}
\def\en{\mathbb N}
\def\eps{\varepsilon}
\def\ov{\overline}
\def \rng {\operatorname{Rng}}
\def \dom {\operatorname{Dom}}
\begin{document}
\author{Marek C\'uth}
\title{Noncommutative Valdivia compacta}
\thanks{The work was supported by the Grant No. 282511/B-MAT/MFF of the Grant Agency of the Charles University in Prague and by the Research grant GA \v{C}R P201/12/0290}
\email{cuthm5am@karlin.mff.cuni.cz}
\address{Charles University, Faculty of Mathematics and Physics, Sokolovsk\'a 83, 186 75 Praha 8 Karl\'{\i}n, Czech Republic}
\begin{abstract}
We prove some generalizations of results concerning Valdivia compact spaces (equivalently spaces with a commutative retractional skeleton) to the spaces with a retractional skeleton (not necessarily commutative). Namely, we show that the dual unit ball of a Banach space is Corson provided the dual unit ball of every equivalent norm has a retractional skeleton. Another result to be mentioned is the following. Having a compact space $K$, we show that $K$ is Corson if and only if every continuous image of $K$ has a retractional skeleton.

We also present some open problems in this area.
\end{abstract}
\maketitle
\section{Introduction}

It is well known that separable Banach spaces have many nice properties. In particular, any separable Banach space admits an equivalent norm which is locally uniformly convex (see e.g. {\cite[Theorem II.2.6]{DGZ}}) and any separable Banach space admits a Markushevich basis (see e.g. {\cite[Theorem 1.22]{hajek}}). Nonseparable Banach spaces need not have those properties. As an example we may take the space $\ell_\infty$ which does not admit any equivalent locally uniformly convex norm and it does not admit a Markushevich basis either (see e.g. {\cite[Theorem II.7.10]{DGZ}} and {\cite[Theorem 5.12]{hajek}}). However, some nonseparable Banach space share those properties of separable ones. For example, any Hilbert space has a locally uniformly convex norm and admits a Markushevich basis.

Having certain nonseparable Banach space, sometimes it is useful to decompose it into smaller pieces (subspaces). There is a hope that if we glue them together, their properties will be preserved by the nonseparable Banach space we started with.

One possible concept of such a decomposition is a \textit{projectional resolution of the identity} (PRI, for short - see e.g. {\cite[Definition 3.35]{hajek}}). However, in a Banach space of density larger than $\aleph_1$, the existence of a PRI does not tell us much about the structure of the space. There are some ways to solve this problem. One of them is the concept of a \textit{projectional generator} (PG, for short). This is a technical tool from which the existence of a PRI follows (see e.g. {\cite[Theorem 3.42]{hajek}}). Moreover, the existence of a PG has consequences also for the structure of the space (see e.g. {\cite[Theorem 5.44]{hajek}}).

Nevertheless, the concept of a PG is not completely satisfactory as it is quite technical. It seems that the right notion is that of a \textit{projectional skeleton} introduced by W.Kubi\'s in \cite{kubis}. The existence of a 1-projectional skeleton implies the existence of a PRI and it has some consequences for the structure of the space (see e.g. {\cite[Corollary 25 and Proposition 26]{kubis}}). Moreover, this notion is not so technical as the concept of a projectional generator.

Spaces with a projectional skeleton are more general than Plichko spaces, but closely connected with them.
Similarly, in \cite{kubisSmall} there has been introduced a class of compact spaces with a retractional skeleton and it has been observed in \cite{kubisSmall} and \cite{kubis} that those spaces are more general than Valdivia compacta, but they share a lot of properties with them.

Motivated by the above, we wanted to see how many properties are preserved and we have generalized some results concerning Valdivia compacta and 1-Plichko spaces.

Namely, we show that the dual unit ball of a Banach space is Corson provided the dual unit ball of every equivalent norm has a retractional skeleton. This generalizes the result contained in \cite{kalenda}.
Another result to be mentioned is the following. Having a compact space $K$ which is a continuous image of a space with a retractional skeleton, we show that the dual unit ball of $\C(K)$ is Corson whenever the dual unit ball of every subspace of $\C(K)$ has a retractional skeleton. This generalizes the result from \cite{kalenda1}.

Proofs of those main results are analogous to the proofs from \cite{hajek} and \cite{kalenda1}. We only had to use some conclusions from \cite{kubis}, \cite{kubisKniha} and apply them. However, three times we had to come with another approach when proving some auxiliary results (see Lemma \ref{lClosedSubset}, \ref{lQuotient} and \ref{lQuotient2}).

Nonetheless, for some statements concerning Valdivia compact spaces we were unable to give similar results concerning spaces with a retractional skeleton. Some of those problems are formulated at the end of this article.\\

Below we recall the most relevant notions, definitions and notations.

We denote by $\omega$ the set of all natural numbers (including $0$), by $\en$ the set $\omega\setminus\{0\}$. Whenever we say that a set is countable, we mean that the set is either finite or infinite and countable. If $f$ is a mapping then we denote by $\rng f$ the range of $f$ and by $\dom f$ the domain of $f$.

Let $T$ be a topological space. The closure of a set $A$ we denote by $\ov{A}$. We say that $A\subset T$ is \textit{countably closed} if $\ov{C}\subset A$ for every countable $C\subset A$. A topological space $T$ is a \textit{Fr\'echet-Urysohn space} if for every $A\subset T$ and every $x\in\ov{A}$ there is a sequence $x_n\in A$ with $x_n\to x$.

All compact spaces are assumed to be Hausdorff. Let $K$ be a compact space. By $\C(K)$ we denote the space of continuous functions on $K$. $P(K)$ stands for the space of probability measures with the $w^*$--topology (the $w^*$--topology is taken from the representation of $P(K)$ as a compact subset of $(\C(K)^*,w^*)$).

Let $\Gamma$ be a set. We put $\Sigma(\Gamma) = \{x\in\er^\Gamma:\;|\{\gamma\in\Gamma:\;x(\gamma)\neq 0\}|\leq\omega\}$. Given a compact $K$, $A\subset K$ is called a \textit{$\Sigma$-subset} of $K$ if there is a homeomorphic embedding $h:K\to[0,1]^\kappa$ such that $A = h^{-1}[\Sigma(\kappa)]$. A compact space $K$ is \textit{Corson compact} if $K$ is a $\Sigma$-subset of $K$. A compact space $K$ is \textit{Valdivia compact} if there exists a dense $\Sigma$-subset of $K$.

We shall consider Banach spaces over the field of real numbers (but many results hold for complex spaces as well). If $X$ is a Banach space and $A\subset X$, we denote by $\conv{A}$ the convex hull of $A$. $B_X$ is the unit ball in $X$ (i.e. the set $\{x\in X:\; \|x\| \leq 1\}$). $X^*$ stands for the (continuous) dual space of $X$. For a set $A\subset X^*$ we denote by $\ov{A}^{w^*}$ the $weak^*$ closure of $A$.

A set $D\subset X^*$ is r-\textit{norming} if
$$\|x\| \leq r. \sup\{|x^*(x)|:\;x^*\in D\cap B_{X^*}\}.$$
We say that a set $D\subset X^*$ is norming if it is r-norming for some $r\geq 1$.

Recall that a Banach space $X$ is called \textit{Plichko} (resp. 1-\textit{Plichko}) if there are a linearly dense set $M\subset X$ and a norming (resp. 1-norming) set $D\subset X^*$ such that for every $x^*\in D$ the set $\{m\in M:\;x^*(m)\neq 0\}$ is countable. 

\begin{definition}A \textit{projectional skeleton} in a Banach space $X$ is a family of projections $\{P_s\}_{s\in\Gamma}$, indexed by an up-directed partially ordered set $\Gamma$, such that
\begin{enumerate}[\upshape (i)]
	\item $X = \bigcup_{s\in\Gamma}P_s X$ and each $P_s X$ is separable.
	\item $s\leq t \Rightarrow P_s = P_s\circ P_t = P_t\circ P_s.$
	\item Given $s_0 < s_1 < \cdots$ in $\Gamma$, $t = \sup_{n\in\omega}s_n$ exists and $P_t X = \ov{\bigcup_{n\in\omega}P_{s_n}X}$.
\end{enumerate}
We shall say that $\{P_s\}_{s\in\Gamma}$ is an \textit{r-projectional skeleton} if it is a projectional skeleton such that $\|P_s\|\leq r$ for every $s\in\Gamma$.

We say that $\{P_s\}_{s\in\Gamma}$ is a \textit{commutative projectional skeleton} if $P_s\circ P_t = P_t\circ P_s$ for every $s,t\in\Gamma$.
\end{definition}

\begin{definition}Let $\mathfrak{s} = \{P_s\}_{s\in\Gamma}$ be a projectional skeleton in a Banach space $X$ and let $D(\mathfrak{s}) = \bigcup_{s\in\Gamma}P^*_s[X^*]$. Then we say that $D(\mathfrak{s})$ is induced by a projectional skeleton.\end{definition}

Recall that due to \cite{kubis}, we may always assume that every projectional skeleton is an r-projectional skeleton for some $r\geq 1$ (just by passing to a suitable cofinal subset of $\Gamma$).

\begin{definition}A \textit{retractional skeleton} ($r$-\textit{skeleton}, for short) in a compact space $K$ is a family of retractions $\{r_s\}_{s\in\Gamma}$, indexed by an up-directed partially ordered set $\Gamma$, such that
\begin{enumerate}[\upshape (i)]
	\item For every $x\in K$, $x = \lim_{s\in\Gamma}r_s(x)$ and $r_s[K]$ is metrizable for each $s\in\Gamma$.
	\item $s\leq t \Rightarrow r_s = r_s\circ r_t = r_t\circ r_s.$
	\item Given $s_0 < s_1 < \cdots$ in $\Gamma$, $t = \sup_{n\in\omega}s_n$ exists and $r_t(x) = \lim_{n\to\infty}r_{s_n}(x)$ for every $x\in K$.
\end{enumerate}
We shall say that $\{r_s\}_{s\in\Gamma}$ is a \textit{commutative retractional skeleton} if $r_s\circ r_t = r_t\circ r_s$ for every $s,t\in\Gamma$.
\end{definition}

By $\R_0$ we denote the class of all compacta which have a retractional skeleton.

\begin{definition}Let $\mathfrak{s} = \{r_s\}_{s\in\Gamma}$ be a retractional skeleton in a compact space $K$ and let $D(\mathfrak{s}) = \bigcup_{s\in\Gamma}r_s[K]$. Then we say that $D(\mathfrak{s})$ is induced by an $r$-skeleton in $K$.\end{definition}

The class of Banach spaces with a projectional skeleton (resp. class of compact spaces with a retractional skeleton) is closely related to the concept of Plichko spaces (resp. Valdivia compacta). By {\cite[Theorem 27]{kubis}}, Plichko spaces are exactly spaces with a commutative projectional skeleton. By {\cite[Theorem 6.1]{kubisSmall}}, Valdivia compact spaces are exactly compact spaces with a commutative retractional skeleton. Moreover, it immediately follows from the proof of {\cite[Theorem 6.1]{kubisSmall}} that whenever $K$ is a Valdivia compact with a dense $\Sigma$-subset $A$, then $A$ is induced by a commutative $r$-skeleton in $K$.

When $X$ is a Banach space with $\operatorname{dens} X = \aleph_1$, then $X$ is a Plichko space if and only if it has a projectional skeleton. Similarly, if $K$ is a compact space with weight $\leq\aleph_1$, then $K$ is Valdivia if and only if $K$ has a retractional skeleton. Indeed, in this case the projectional (resp. retractional) skeleton can be indexed by a well-ordered set $[0,\aleph_1)$, so it may be commutative.

An example of a compact space with a retractional skeleton which is not Valdivia is $[0,\omega_2]$ (see {\cite[Example 6.4]{kubisSmall}}). An example of a space with a 1-projectional skeleton which is not Plichko is $\C([0,\omega_2])$ (see {\cite[Theorem 1]{kalenda4}}).

\section{Main results}
The following is a generalization of {\cite[Theorem 1]{kalenda}}.
\begin{thm}\label{tMain1}The following conditions are equivalent for a Banach space $\langle X,\|\cdot\|\rangle$:
 \begin{enumerate}[\upshape (i)]
	\item $(B_{\langle X^*,\|\cdot\|\rangle},w^*)$ is Corson.	
	\item $\langle X,\tn\cdot\tn\rangle$ has a 1-projectional skeleton for every equivalent norm $\tn\cdot\tn$.
	\item $(B_{\langle X^*,\tn\cdot\tn\rangle},w^*)$ has a retractional skeleton for every equivalent norm $\tn\cdot\tn$.
 \end{enumerate}
\end{thm}

The following is a generalization of {\cite[Theorem 3.1]{kalenda2}}.

\begin{thm}\label{tMain}The following conditions are equivalent for a compact space $K$.
	\begin{enumerate}[\upshape (i)]
		\item $K$ is a Corson compact.
		\item Every continuous image of $K$ has a retractional skeleton.
	\end{enumerate}
\end{thm}

We will get the last theorem as a special case of Theorem \ref{tMain3} bellow. To formulate it in a simple general way, we use the class of compact spaces introduced in \cite{kalenda1}.

\begin{definition}
  A compact Hausdorff space is said to belong to the class $\GO$ if for every nonempty open subset $U\subset K$ the following holds.
 \begin{quotation}If $U$ does not contain at least one $G_\delta$ point of $K$, then the one-point compactification of $U$ contains a homeomorphic copy of $[0,\omega_1]$.\end{quotation}
\end{definition}

We will need also the following notion of property $(M)$.

\begin{definition}
 A compact space $K$ is said to have the property $(M)$ if every Radon probability measure on $K$ has separable support.
\end{definition}

The following two theorems are generalizations of {\cite[Theorem 1]{kalenda1}} and {\cite[Theorem 2]{kalenda1}}.
\begin{thm}\label{tMain2}The following conditions are equivalent for a compact space $K$ from the class $\GO$.
 \begin{enumerate}[\upshape (i)]
 	\item $K$ is a Corson compact with the property $(M)$.
 	\item Every subspace of $\C(K)$ has a 1-projectional skeleton.
 	\item $(B_{Y^*},w^*)$ has a retractional skeleton for every subspace $Y\subset \C(K)$.
 \end{enumerate}
 In particular, the assumptions of this theorem are satisfied if $K$ is a continuous image of a space with a retractional skeleton.
\end{thm}

\begin{thm}\label{tMain3}The following conditions are equivalent for a compact space $K$ from the class $\GO$.
 \begin{enumerate}[\upshape (i)]
 	\item $K$ is a Corson compact.
 	\item $\C(L)$ has a 1-projectional skeleton for every continuous image $L$ of $K$.
 	\item $(B_{\C(L)^*},w^*)$ has a retractional skeleton for every continuous image $L$ of $K$.
 	\item $P(L)$ has a retractional skeleton for every continuous image $L$ of $K$.
 \end{enumerate}
 In particular, the assumptions of this theorem are satisfied if $K$ is a continuous image of a space with a retractional skeleton.
\end{thm}

\section{Properties of compact spaces with a retractional skeleton}

In this section we first collect several important properties of sets induced by an $r$-skeleton in a compact space. Those properties are similar to properties of dense $\Sigma$-subsets of Valdivia compact spaces and the proofs are often done in a similar way.
Having those results in hand, we deduce from them some properties of compact spaces with an $r$-skeleton. Those are similar to the ones of Valdivia compact spaces and the proofs are often done in the same way.

We start with the following theorem which sums up basic properties of sets induced by an $r$-skeleton.

\begin{thm}[{\cite[Theorem 30]{kubis}}]\label{tRetractFrechet}Assume $D$ is induced by an $r$-skeleton in $K$. Then:
\begin{enumerate}[\upshape (i)]
	\item $D$ is dense in $K$ and for every countable set $A\subset D$, $\overline{A}$ is metrizable and contained in $D$.
	\item $D$ is a Fr\'echet-Urysohn space.
	\item $D$ is a normal space and $K = \beta D$.
\end{enumerate}
\end{thm}

We continue with some consequences of Theorem \ref{tRetractFrechet}. The following lemma is just an easy generalization of {\cite[Lemma 1.7]{kalendaSurvey}}. The proof is identical, we only use Theorem \ref{tRetractFrechet} instead of {\cite[Lemma 1.6]{kalendaSurvey}}.
\begin{lemma}\label{lUnique}Let $K$ be a compact space and $A,B$ be two subsets induced by an $r$-skeleton in $K$. If $M\subset K$ is a set such that $A\cap B\cap M$ is dense in $M$, then $A\cap M = B\cap M$. In particular, $A=B$ whenever $A\cap B$ is dense in $K$.\end{lemma}
As any set induced by an $r$-skeleton in a compact space is countably compact, we get as a consequence of {\cite[Lemma 1.11]{kalendaSurvey}} the following.
\begin{lemma}\label{lGDelta}Assume $D$ is induced by an $r$-skeleton in a compact space $K$. Then $G\cap D$ is dense in $G$ whenever $G\subset K$ is $G_\delta$. In particular, if $x\in K$ is a $G_\delta$ point, then $x\in D$.\end{lemma}
\begin{cor}\label{cUnique}Let $K$ be a compact space with a dense set of $G_\delta$ points. Then there is at most one set $D$ which is induced by an $r$-skeleton in $K$.\end{cor}\begin{proof}This follows immediately from Lemma \ref{lUnique} and Lemma \ref{lGDelta}.\end{proof}

We continue with the following lemma.
\begin{lemma}\label{lClosedSubset}Let $K$ be a compact space, $F\subset K$ closed subset and let $D\subset K$ be such that $D$ is induced by an $r$-skeleton in $K$. If $D\cap F$ is dense in $F$, then $D\cap F$ is induced by an $r$-skeleton in $F$.\end{lemma}
\begin{proof}Let $\mathfrak{s} = \{r_s\}_{s\in\Gamma}$ be a retractional skeleton in $K$ and put
$$\Gamma' = \{s\in\Gamma:\;r_s[F]\subset F\}.$$
In order to see $\mathfrak{s}' = \{r_s\upharpoonright_{F}\}_{s\in\Gamma'}$ is a retractional skeleton in $F$, it is enough to prove that $\Gamma'$ is a cofinal subset of $\Gamma$ such that for every sequence $s_0<s_1<\cdots$ in $\Gamma'$, $\sup_{n\in\omega}s_n\in\Gamma'$. Once this is proved, it is easy to notice that $D(\mathfrak{s}') = D\cap F$.

In order to verify that $\Gamma'$ is a cofinal subset of $\Gamma$, fix $\gamma_0\in\Gamma$ and put $C_{-1} = \emptyset$. We inductively find sequences $\{\gamma_n\}_{n\in\omega}\subset\Gamma$ and $\{C_n\}_{n\in\omega}$ in the following way. Having $\gamma_n$ and $C_{n-1}$, we find a countable set $C_n\subset D\cap F$ such that $r_{\gamma_n}[C_n]$ is dense in $r_{\gamma_n}[D\cap F]$. Then, using (ii) and (iii) from the definition of a retractional skeleton and $C_n\subset D$, we find $\gamma_{n+1} > \gamma_n$ such that $C_n\subset r_{\gamma_{n+1}}[K]$. Put $t = \sup_{n\in\omega}\gamma_n$. Now we will prove that $r_t[D\cap F]\subset D\cap F$.

Fix a metric $\rho$ in the space $r_t[K]$ and a point $x\in D\cap F$. Then, for every $n\in\omega$, we find a point $c_n\in C_n$ satisfying
$$\rho(r_{\gamma_k}(x),r_{\gamma_k}(c_n)) < \frac 1n\;,\quad k\leq n.$$
Such a point exists, because $\{z\in r_{\gamma_n}[D\cap F]:\;\forall k\leq n:\rho(r_{\gamma_k}(x),r_{\gamma_k}(z)) < \frac 1n\}$ is an open set in $r_{\gamma_n}[D\cap F]$ containing $r_{\gamma_n}(x)$; thus, it contains $r_n(c_n)$ for some $c_n\in C_n$.
Passing to a subsequence if necessary, we may without loss of generality assume there is a point $c\in r_t[K]$ such that $c_n\to c$. Consequently, $\rho(r_{\gamma_k}(x),r_{\gamma_k}(c)) = 0$ for every $k\in\omega$. Hence, $$r_t(x) = \lim_{k\to\infty}r_{\gamma_k}(x) = \lim_{k\to\infty}r_{\gamma_k}(c) = r_t(c) = \lim_{n\to\infty}r_t(c_n) = \lim_{n\to\infty}c_n \in D\cap F$$
and $r_t[D\cap F]\subset D\cap F$.

Using the density of $D\cap F$ in $F$, $t\in\Gamma'$ and $\Gamma'$ is cofinal in $\Gamma$.

Having $s_0<s_1<\cdots$ in $\Gamma'$ and $t = \sup_{n\in\omega}s_n$, it is obvious that for every $x\in F$, $r_t(x) = \lim_{n\to\infty}r_{s_n}(x)\in F$. Thus, $t\in\Gamma'$.
\end{proof}
Notice, that the preceding lemma is trivial in the case when $D$ is a dense $\Sigma$-subset of $K$. However, for spaces with a retractional skeleton this required some work. The proof can also be done using the method of elementary submodels, namely Theorem \ref{tRetractModel}. This alternative proof is much shorter, but its difficulty is hidden in Theorem \ref{tRetractModel} and in the method of elementary submodels.

The following lemma is a strengthening of Lemma \ref{lGDelta}. It is just an easy generalization of {\cite[Lemma 5]{kalenda1}}. Every set induced by an $r$-skeleton in a compact space $K$ satisfies (by Theorem \ref{tRetractFrechet} and Lemma \ref{lClosedSubset}) all the properties of dense $\Sigma$-subsets in $K$ which are required in the proof from \cite{kalenda1}. Hence, the proof of the lemma can be done in the same way as the proof of {\cite[Lemma 5]{kalenda1}}.

\begin{lemma}\label{lSubsetStability}Let $K$ be a compact space and $G = \bigcap_{n\in\en}\overline{U}_n$ where each $U_n$ is an open subset of $K$. If $D$ is induced by an $r$-skeleton in $K$, then $G\cap D$ is dense in $G$. Consequently, $G\cap D$ is induced by an $r$-skeleton in $G$.
\end{lemma}

Now we collect several properties of compact spaces with an $r$-skeleton which follow from the above results concerning sets induced by an $r$-skeleton. As an easy corollary to Theorem \ref{tRetractFrechet} we get the following.

\begin{cor}\label{cNotSkeleton}Let $K$ be a compact space, $x\in K$ and $\Gamma$ an uncountable set. Let $\{g_k\}_{k=1}^\infty$ and $\{f_\gamma\}_{\gamma\in\Gamma}$ be sets of $G_\delta$ points in $K$ such that $x\in\overline{\{g_k\}}_{k=1}^\infty\cap\overline{\{f_\gamma\}}_{\gamma\in\Gamma}$ and no countable sequence from $\{f_\gamma\}_{\gamma\in\Gamma}$ converges to $x$. Then $K$ does not have a retractional skeleton.
\end{cor}
\begin{proof}In order to get a contradiction, let us assume that a set $D$ is induced by an $r$-skeleton in $K$. Then $\{g_k\}_{k=1}^\infty\cup \{f_\gamma\}_{\gamma\in\Gamma}\subset D$, and since $x\in\overline{\{g_k\}_{k=1}^\infty}$, it follows that $x\in D$. Put $C = \{f_\gamma\}_{\gamma\in\Gamma}\subset D$. Then $x\in \overline{C}$, but no countable sequence from $C$ converges to $x$. This is a contradiction with the fact that $D$ is a Fr\'echet space.
\end{proof}
In {\cite[Example 3.4]{kalenda2}} there are some basic examples of compact spaces which are not Valdivia. Since they have both weight $\aleph_1$, they do not have a retractional skeleton either. We sum up those in the example bellow.
\begin{exa}\label{eNotRetractSkeleton}\begin{enumerate}
	\item Let $K_1$ be the compact space obtained from $([0,\omega_1]\times\{0\})\oplus([0,\omega]\times\{1\})$ by identifying the points $(\omega_1,0)$ and $(\omega,1)$. Then $K_1\notin\R_0$.
	\item Let $K_2$ be the compact space obtained from $[0,\omega_1]\times\{0,1\}$ by identifying the points $(\omega_1,0)$ and $(\omega_1,1)$. Then $K_2\notin\R_0$.
\end{enumerate}
\end{exa}

The following stability result follows immediately from Lemma \ref{lGDelta}, Lemma \ref{lClosedSubset} and Lemma \ref{lSubsetStability}.
\begin{thm}\label{tDedicnost}Let $K$ be a compact space with a retractional skeleton. Then:\begin{enumerate}[\upshape (i)]
	\item Every subset of $K$, which is the closure of an arbitrary union of $G_\delta$ sets, has a retractional skeleton as well.
	\item If $G = \bigcap_{n\in\en}\overline{U}_n$ with $U_n$ open, then $G$ has a retractional skeleton as well.
\end{enumerate}\end{thm}

We continue with a theorem from \cite{kubisKniha}. First, we recall some definitions (see \cite{kubisKniha}).

We denote by $\R$ the smallest class of compact spaces that contains all metrizable ones and that is closed under limits of continuous retractive inverse sequences. It is claimed in \cite{kubis} that $\R_0\subset \R$. A more detailed proof of this fact is contained in the proof of {\cite[Lemma 6.3]{kubisSmall}} (the assumption on the commutativity of the skeleton is not needed to obtain $\R_0\subset\R$).

We denote by $\RC$ the smallest class of compact spaces that contains all metrizable ones and that is closed under continuous images and inverse limits of transfinite sequences of retractions. Obviously, $\R\subset \RC$.

\begin{thm}[{\cite[Theorem 19.22]{kubisKniha}}]\label{tDichotomy}Let $K\in \RC$. Then either $[0,\omega_1]$ embeds into $K$ or else $K$ is Corson compact.\end{thm}

Now we show what is the correspondence between compact spaces with a retractional skeleton and Corson compact spaces.
\begin{thm}\label{tCorson}Let $K$ be a compact space. Then it is a Corson compact if and only if $K$ is induced by an $r$-skeleton in $K$. Moreover, whenever $D$ is a set induced by an $r$-skeleton in a Corson compact $K$, then $D = K$.
\end{thm}
\begin{proof}Let $K$ be a Corson compact. Then, as mentioned above, it immediately follows from the proof of {\cite[Theorem 6.1]{kubisSmall}} that $K$ is induced by a commutative $r$-skeleton in $K$. Moreover, whenever $D$ is induced by an $r$-skeleton in $K$, $D=K$ by Lemma \ref{lUnique}.

If $K$ is induced by an $r$-skeleton in $K$, then $K$ is Fr\'echet-Urysohn space; thus, $[0,\omega_1]$ does not embed into $K$. It follows from Theorem \ref{tDichotomy} that $K$ is Corson.
\end{proof}
\begin{cor}\label{cSubsetCorson}Assume $D$ is induced by an $r$-skeleton in a compact space $K$. Then every subset of $D$ closed in $K$ is Corson.\end{cor}
\begin{proof}This follows from Lemma \ref{lClosedSubset} and from Theorem \ref{tCorson} above.\end{proof}

The following lemma is an analogue to {\cite[Lemma 2]{kalenda1}}.

\begin{lemma}\label{lInPKSpaces}Let $K$ be a compact space such that $P(K)$ has a retractional skeleton. If we denote by $G$ the set of all $G_\delta$ points of $K$, then $\ov{G}$ has a retractional skeleton as well.\end{lemma}
\begin{proof}We use the same idea as in {\cite[Lemma 2]{kalenda1}}. Let us fix a set $D$ induced by an $r$-skeleton in $P(K)$. If $g\in G$ is a $G_\delta$ point of $K$, then it is easy to verify (see {\cite[Lemma 5.5]{kalendaSurvey}}), that the Dirac measure $\delta_g$ supported by the point $g$ is a $G_\delta$ point in $P(K)$; hence, by Lemma \ref{lGDelta}, $\delta_g\in D$. Thus, if we identify $k$ with $\delta_k$ for every $k\in K$, $G\subset D$. Consequently, by Lemma \ref{lClosedSubset}, $\ov{G}\cap D$ is induced by an $r$-skeleton in $\overline{G}$.\end{proof}

\begin{proposition}\label{lPSkeletonImpliesRSkeleton}Let $D\subset X^*$ be a set induced by a 1-projectional skeleton. Then there exists a convex symmetric set $R$, induced by an $r$-skeleton in $(B_{X^*},w^*)$.
\end{proposition}
\begin{proof}Let $\{P_s\}_{s\in\Gamma}$ be a 1-projectional skeleton such that $D = \bigcup_{s\in\Gamma}P^*_s(X^*)$. Using only the definitions and {\cite[Lemma 10]{kubis}} it is easy to see that $\{P^*_s\upharpoonright_{B_{X^*}}\}_{s\in\Gamma}$ is retractional skeleton in $B_{X^*}$. Now it remains to show that $R = \bigcup_{s\in\Gamma}P^*_s(B_{X^*})$ is convex and symmetric. It is easily checked that $R = D\cap B_{X^*}$. Now we observe that $D$ is an up-directed union of linear sets; thus, it is linear. Consequently, $R$ is convex and symmetric.
\end{proof}

Now we are ready to see that once we know (i)$\Rightarrow$(ii) in Theorem \ref{tMain2} (resp. Theorem \ref{tMain3}), (iii) (resp. (iv)) is the strongest condition.
\begin{proposition}\label{pSkeletonAndPK}Let $K$ be a compact space. Consider the following conditions
\begin{enumerate}[\upshape (i)]
	\item $K$ has a retractional skeleton.
	\item $\C(K)$ has a 1-projectional skeleton.
	\item There is a convex symmetric set induced by an $r$-skeleton in $(B_{\C(K)^*},w^*)$.
	\item $(B_{\C(K)^*},w^*)$ has a retractional skeleton.
	\item There is a convex symmetric set induced by an $r$-skeleton in $P(K)$.
	\item $P(K)$ has a retractional skeleton.
\end{enumerate}Then the following implications hold:
$$(i)\Rightarrow (ii)\Rightarrow (iii)\Rightarrow (iv)\Rightarrow (vi),\quad(iii)\Rightarrow (v)\Rightarrow (vi).$$
Moreover, if $K$ has a dense set of $G_\delta$ points, then all the conditions are equivalent.
\end{proposition}
\begin{proof}The implication (i)$\Rightarrow$(ii) comes from {\cite[Proposition 28]{kubis}}, (ii)$\Rightarrow$(iii) follows from Lemma \ref{lPSkeletonImpliesRSkeleton}, (iii)$\Rightarrow$(iv) and (v)$\Rightarrow$(vi) are obvious. For (iv)$\Rightarrow$(vi) and (iii)$\Rightarrow$(v) it is enough to observe that $P(K)$ is a closed $G_\delta$ set in $(B_{\C(K)^*},w^*)$ and use Lemma \ref{lGDelta} and Lemma \ref{lClosedSubset}. If $K$ has a dense set of $G_\delta$ points, then (vi)$\Rightarrow$(i) follows from Lemma \ref{lInPKSpaces}.
\end{proof}

\begin{remark}\label{remark}It is known that the implication (ii)$\Rightarrow$(i) in Proposition \ref{pSkeletonAndPK} does not hold. There even exists a compact space $K$ such that $\C(K)$ is 1-Plichko, but $K\notin\R$ (see {\cite[Theorem 3.2]{kubisExample}}). However, in the case of commutative skeletons (i.e. Plichko spaces and Valdivia compact spaces), it is true that (ii)$\Leftrightarrow$(iii)$\Leftrightarrow$(v) (see {\cite[Theorem 5.2]{kalendaSurvey}}). The proof of this fact uses a characterization of dense subsets induced by an $r$-skeleton in $K$ by a topological property of the space $(\C(K),\tau_p(D))$ ($\tau_p(D)$ is the topology of the pointwise convergence on $D$). Thus, two natural questions arise. They are formulated at the end of this article (see Problem \ref{problem1} and Question \ref{question1}).
\end{remark}

\section{The method of elementary submodels}

In this section we prove Lemma \ref{lQuotient} and Lemma \ref{lQuotient2} using the method of elementary submodels. If one does not feel comfortable with this method, he can skip this section and use only its results. The knowledge of the method elementary submodels is not needed any further.

The method of elementary submodels is a set-theoretical method which can be used in various branches of mathematics.
W.Kubi\'s in \cite{kubis} used this method to create a projectional (resp. retractional) skeleton in certain Banach (resp. compact) spaces. In \cite{cuth} the method has been slightly simplified and precised. We briefly recall some basic facts about the method and give a more detailed proof of Theorem \ref{tRetractModel} which is also proved in a slightly different form in {\cite[Theorem 19.16]{kubisKniha}}. Finally, we prove Lemma \ref{lQuotient} and Lemma \ref{lQuotient2}.

First, let us recall some definitions:

Let $N$ be a fixed set and $\phi$ a formula in the language of $ZFC$. Then the {\em relativization of $\phi$ to $N$} is the formula $\phi^N$ which is obtained from $\phi$ by replacing each quantifier of the form ``$\forall x$'' by ``$\forall x\in N$'' and each quantifier of the form ``$\exists x$'' by ``$\exists x\in N$''.

If $\phi(x_1,\ldots,x_n)$ is a formula with all free variables shown (i.e. a formula whose free variables are exactly $x_1,\ldots,x_n$) then {\em $\phi$ is absolute for $N$} if and only if
$$\forall a_1,\ldots,a_n\in N\quad (\phi^N(a_1,\ldots,a_n) \leftrightarrow \phi(a_1,\ldots,a_n)).$$

The method is based mainly on the following theorem (a proof can be found in {\cite[Theorem IV.7.8]{kunen}}).

\begin{thm}\label{tCountModel}
 Let $\phi_1,\ldots,\phi_n$ be any formulas and $X$ any set. Then there exists a set $M\supset X$ such, that
 $$(\phi_1,\ldots,\phi_n \text{ are absolute for }M)\quad \wedge\quad (|M|\leq \max(\omega,|X|)).$$
\end{thm}

Since the set from Theorem \ref{tCountModel} will often be used, the following notation is useful.

\begin{definition}
 Let $\phi_1,\ldots,\phi_n$ be any formulas and let $X$ be any countable set. Let $M\supset X$ be a countable set satisfying that $\phi_1,\ldots,\phi_n$ are absolute for $M$. Then we say that {\em $M$ is an elementary submodel for $\phi_1,\ldots,\phi_n$ containing $X$}. This is denoted by $M\prec(\phi_1,...,\phi_n;\; X)$.
\end{definition}

We shall also use the following convention.
\begin{convention} \label{conventionM}
 Whenever we say\\[8pt]
 {\em for any suitable elementary submodel $M$ (the following holds...)},\\[8 pt]
 we mean that\\[8pt]
 {\em there exists a list of formulas $\phi_1,\ldots,\phi_n$ and a countable set $Y$ such that for every $M\prec(\phi_1,\ldots,\phi_n;\;Y)$ (the following holds...)}.
\end{convention}

By using this new terminology we lose the information about the formulas $\phi_1,\ldots,\phi_n$ and the set $Y$. This is, however, not important in applications.

Let us recall several more results about suitable elementary submodels (proofs can be found in {\cite[Chapters 2 and 3]{cuth}}):
\begin{lemma}\label{lCupM}
 Let $\varphi_1,\ldots,\varphi_n$ be a subformula closed list of formulas and let $X$ be any countable set. Let $\{M_k\}_{k\in\omega}$ be a sequence of sets satisfying
 \begin{itemize}
	\item[(i)] $M_i\subset M_j,\quad i\leq j,$
	\item[(ii)] $\forall k\in\omega:\; M_k\prec(\varphi_1,...,\varphi_n;\; X).$	
 \end{itemize}
 Then for $M:=\bigcup_{k\in\omega}{M_k}$ it is true, that also $M \prec(\varphi_1,...,\varphi_n;\; X)$.
\end{lemma}
\begin{lemma}\label{lBasicPropertiesOfM}For any suitable elementary submodel $M$ the following holds:
  \begin{enumerate}[\upshape (i)]
    \item Let $f$ be a function such that $f\in M$. Then $\dom{f} \in M$, $\rng{f} \in M$ and $f(M)\subset M$.
    \item Let $S\in M$ be a countable set. Then $S\subset M$.
  \end{enumerate}
\end{lemma}
\begin{lemma}\label{lUniqueInM}
 Let $\phi(y,x_1,\ldots,x_n)$ be a formula with all free variables shown and let $X$ be a countable set. Let $M$ be a fixed set, $M\prec(\phi, \exists y\phi(y,x_1,\ldots,x_n);\; X)$ and let $a_1,\ldots,a_n \in M$ be such that there exists only one set $u$ satisfying $\phi(u,a_1,\ldots,a_n)$. Then $u\in M$.
\end{lemma}
Using the last lemma we can force the elementary submodel $M$ to contain all the needed objects created (uniquely) from elements of $M$.

Given a compact space $K$ and an arbitrary elementary submodel $M$ we define the following equivalence relation $\sim_M$ on $K$:
$$x\sim_M y \quad\iff\quad (\forall f \in \C(K)\cap M):\; f(x) = f(y).$$
We shall write $K/_M$ instead of $K/_{\sim_M}$ and we shall denote by $q_K^M$ the canonical quotient map. It is not hard to check that $K/_M$ is a compact Hausdorff space.

In \cite{cuth} it is proved that the following lemma holds (slightly different version may be also found in \cite{kubis}). 

\begin{lemma}\label{lCKM}Let $K$ be a compact space. Then for any suitable elementary submodel $M$ it is true that
  $$\ov{\C(K)\cap M} = \{\varphi\circ q_K^M:\;\varphi\in \C(K/_M)\}.$$
Consequently, we can identify $\ov{\C(K)\cap M}$ with the space $\C(K/_M)$.
\end{lemma}

We will need the following simple, but useful lemma.

\begin{lemma}\label{lMetrizableCpct}For every suitable elementary submodel $M$ the following holds: Let $K$ be a compact metric space. Then whenever $K\in M$, $\C(K)\cap M$ separates points of $K$.
\end{lemma}
\begin{proof}Fix a suitable elementary submodel $M$ such that $K\in M$. Then, using the absoluteness of the formula (and its subformula)
$$\exists D(D\text{ is a countable subset of }\C(K)\text{ separating points of }K),$$
there exists a countable set $D\in M$ separating points of $K$. By Lemma \ref{lBasicPropertiesOfM}, $D\subset M$. Consequently, $\C(K)\cap M \supset D$ separates points of $K$.
\end{proof}

Finally, let us show how the method of elementary submodels is connected with the compact spaces with a retractional skeleton.

\begin{thm}[{\cite[Theorem 19.16]{kubisKniha}}]\label{tRetractModel}Let $K$ be a compact space, and let $D$ be its dense subset. The following properties are equivalent:
\begin{enumerate}[\upshape (i)]
	\item There exists a set $D(\mathfrak{s})$ induced by an $r$-skeleton in $K$ such that $D\subset D(\mathfrak{s})$.
	\item For every	suitable elementary submodel $M$, the quotient map $q_K^M:K\to K/_M$ is one-to-one on $\ov{D\cap M}$.
\end{enumerate}
Under the assumption that $D$ is countably closed, the conditions above are also equivalent to the following:
\begin{enumerate}[\upshape (iii)]
	\item $D$ is induced by an $r$-skeleton in $K$.
\end{enumerate}
\end{thm}
\begin{proof}
First, let us suppose that (i) holds. Without loss of generality we assume that $D = D(\mathfrak{s})$ is induced by an $r$-skeleton $\{r_s\}_{s\in\Gamma}$ in $K$. Define a mapping $r:\Gamma\to\C(K)$ by $r(s) = r_s$. Fix formulas $\varphi_1,\ldots,\varphi_n$ containing all the formulas (and their subformulas) marked by $(*)$ in the proof below and a countable set $Y\supset \{D,K,\Gamma,r\}$ such that whenever $M\prec(\varphi_1,...,\varphi_n;\; Y)$, all the results mentioned above hold for $M$. Fix some $x,y\in\ov{D\cap M}$, $x\neq y$. Using the the absoluteness of the formula (and its subformulas)
\begin{flalign*}
    (*) & & \forall u,v\in \Gamma\;\exists w\in\Gamma\;w\geq u,v, & &
 \end{flalign*}
the set $(\Gamma\cap M)$ is up-directed. Thus, there exists $t = \sup(\Gamma\cap M)$. Using the absoluteness of the formula (and its subformulas)
\begin{flalign*}
    (*) & & \forall x\in D\;\exists s\in\Gamma\;x\in r_s[K], & &
 \end{flalign*}
$D\cap M\subset r_t[K]$. Thus, $\ov{D\cap M}\subset r_t[K]$. Now we find a sequence $s_0 < s_1 \cdots$ in $\Gamma\cap M$ such that $\sup_{n\in\omega}s_n = t$. Then $x = \lim_{n\to\infty} r_{s_n}(x)$ and $y = \lim_{n\to\infty} r_{s_n}(y)$. There exists an $n\in\en$ such that $r_{s_n}(x)\neq r_{s_n}(y)$. By Lemma \ref{lMetrizableCpct}, there is a function $f\in \C(r_{s_n}[K])\cap M$ such that $f(r_{s_n}(x))\neq f(r_{s_n}(y))$. By Lemma \ref{lBasicPropertiesOfM}, $r(s_n) = r_{s_n}\in M$. Now, using the absoluteness of the formula (and its subformula)
\begin{flalign*}
    (*) & & \forall f,g\in \C(K)\;\exists h\in\C(K)\quad(h = f\circ g), & &
 \end{flalign*}
$g = f\circ r_{s_n}\in \C(K)\cap M$ and $g(x)\neq g(y)$.

In order to prove (ii)$\Rightarrow$(i), fix formulas $\varphi_1,\ldots,\varphi_n$ and a countable set $Y$ such that whenever $M\prec(\varphi_1,...,\varphi_n;\; Y)$, $q_K^M$ is one-to-one on $\ov{D\cap M}$, all the statements mentioned above about suitable models hold and all the formulas (and their subformulas) marked by $(*)$ below are absolute for $M$. We can without loss of generality assume that $K,D\in Y$ (if not, we just put $Y' = Y\cup \{K,D\}$). Fix $M\prec(\varphi_1,...,\varphi_n;\; Y)$. In the following we write $q^M$ instead of $q_K^M$. Observe that $q^M[D\cap M]$ is a dense subset of $K/_M$.

Indeed, using Lemma \ref{lCKM} it is not difficult to show that
$$\{\psi^{-1}(U):\;\psi\in\C(K/_M),\;\psi\circ q^M\in\C(K)\cap M,\;U\text{ is an open rational interval}\}$$
is a basis of $K/_M$. Now if we take an open rational interval $U$ and a function $\psi\in\C(K/_M)$ such that $\psi\circ q^M\in\C(K)\cap M$, then by the denseness of $q[D]$ in $K/_M$ there is a $d\in D$ such that $q^M(d)\in \psi^{-1}(U)$. Thus, \begin{flalign*}
    (*) & & \exists d\in D\quad\psi(q^M(d))\in U. & &
 \end{flalign*}
 Using the elementarity of $M$, there is a $d\in D\cap M$ such that $\psi(q^M(d))\in U$. Hence, $q^M[D\cap M]\cap \psi^{-1}(U)\neq\emptyset$.

It follows that $q^M[\ov{D\cap M}] = K/_M$. If we denote $j^M = (q^M\upharpoonright_{\ov{D\cap M}})^{-1}$, then $j^M$ is a homeomorphism of $\ov{D\cap M}$ and $K/_M$. It follows that $r_M = j^M\circ q^M:K\to\ov{D\cap M}$ is a retraction onto.

By Theorem \ref{tCountModel}, there exists a set $R\supset Y\cup K\cup\{U:\;U\text{ is an open set in }K\}$ such that $\varphi_1,...,\varphi_n$ are absolute for $R$. It follows from the proof of Theorem \ref{tCountModel} (see {\cite[Theorem IV.7.8]{kunen}}) that for every countable set $Z\subset R$ there exists an $M\subset R$ such that $M\prec(\varphi_1,...,\varphi_n;\; Z)$. Hence, by Lemma \ref{lCupM},
$$\Gamma = \{M\subset R:\;M\prec(\varphi_1,...,\varphi_n;\; Y)\}$$
is a nonempty and up-directed set where the supremum of every increasing countable chain exists. We will verify that $\{r_M\}_{M\in\Gamma}$ is the retractional skeleton we are looking for.

Observe that $f(r_M(x)) = f(x)$ for every $f\in\ov{\C(K)\cap M}$ and $x\in K$. Indeed, every $f\in\ov{\C(K)\cap M}$ equals $\psi\circ q^M$ for some $\psi\in\C(K/_M)$. It follows that for every $x\in K$
$$f(r_M(x)) = \psi (q^Mr_M(x)) = \psi (q^Mj^Mq^M(x)) = \psi (q^M(x)) = f(x).$$
Moreover, as $q^M$ is one-to-one on $\ov{D\cap M}$, $\C(K)\cap M$ separates points of $\ov{D\cap M}$.

Fix some $M\in\Gamma$. The set $r_M[K] = \ov{D\cap M}$ is homeomorphic to $K/_M$; hence, it is metrizable. In order to verify (i) from the definition of a retractional skeleton, fix $x\in K$ and an open set $U\ni x$. Find $M\in\Gamma$ such that $x,U\in M$. Using the absoluteness of the formula (and its subformula)
\begin{flalign*}
    (*) & & \exists f\in\C(K)\quad(f(x) = 0\;\wedge\;\forall y\in U\;f(y) = 1), & &
 \end{flalign*}
for every $M\subset N\in\Gamma$ there is $f\in\C(K)\cap N$ such that $f(x) = 0$ and $f(y) = 1$ for $y\notin U$. Find a point $d\in\ov{D\cap N}$ such that $q^N(d) = q^N(x)$. Then $r_N(x) = d\in U$ (otherwise $f(r_N(x)) = 1$ which would be a contradiction because $f(r_N(x)) = f(x)$). Consequently, $x = \lim_{M\in\Gamma}r_M(x)$.

To verify (ii) from the definition of a retractional skeleton, fix $M\subset N$ from $\Gamma$. Then it is obvious that $r_N(r_M(x)) = r_M(x)$. Let us take a function $g\in\ov{\C(K)\cap M}\subset \ov{\C(K)\cap N}$ and a point $x\in K$. Then, by the argument above,
$$g(r_M(x)) = g(x) = g(r_N(x)) = g(r_M(r_N(x))).$$
As $\C(K)\cap M$ separates points of $\ov{D\cap M}$, $r_M(x) = r_M(r_N(x)$ holds as well.

Finally, take $M_0\subset M_1\subset \ldots$ in $\Gamma$, $M = \bigcup_{n\in\omega}M_n$ and $x\in K$. Fix $f\in\C(K)\cap M$ and find $n\in\en$ such that $f\in\C(K)\cap M_n$. It follows that for every $k\geq n$, $f(r_M(x)) = f(x) = f(r_{M_k}(x))$. Consequently, $\lim_{n\to\infty}f(r_{M_n}(x)) = f(r_M(x))$ for every $f\in\C(K)\cap M$; hence, for every $f\in\ov{\C(K)\cap M}$. By Lemma \ref{lCKM} and the fact that $\ov{D\cap M}$ is homeomorphic with $K/_M$, we may identify $\ov{\C(K)\cap M}$ with $\C(\ov{D\cap M})$ and $\lim_{n\to\infty}f(r_{M_n}(x)) = f(r_M(x))$ for every $f\in\C(\ov{D\cap M})$. It follows that $\lim_{n\to\infty} r_{M_n}(x) = r_M(x)$.

We have verified that $\mathfrak{s} = \{r_M\}_{M\in\Gamma}$ is a retractional skeleton. Obviously, $D(\mathfrak{s}) = \bigcup_{M\in\Gamma}\ov{D\cap M}\supset D$ and $D = D(\mathfrak{s})$ if $D$ is a countably closed set.
\end{proof}
We end this section with two lemmas. Those are similar statements to {\cite[Lemma 2.8]{kalenda2}} and {\cite[Lemma 6]{kalenda1}}. In proofs we use the method of elementary submodels (namely Theorem \ref{tRetractModel}).
\begin{lemma}\label{lQuotient}Let $K$ be a compact space and $F\subset K$ be a metrizable closed set. Put $L = K\setminus F\cup\{F\}$ endowed with the quotient topology induced by the mapping $Q:K\to L$ defined by
$$Q(x) = \left\{
\begin{array}{ll}
	x & x\in K\setminus F\\
	F & x\in F.
\end{array}\right.$$
If $D$ is induced by an $r$-skeleton in $L$ and $Q^{-1}(D)$ is dense in $K$, then $Q^{-1}(D)$ is induced by an $r$-skeleton in $K$.\end{lemma}
\begin{proof}Let us fix a suitable elementary submodel $M$ such that $Q,K,F\in M$. Notice that $Q^{-1}(D)$ is countably closed. Thus, by Theorem \ref{tRetractModel}, it is enough to verify that $q_K^M$ is one-to-one on $\overline{Q^{-1}(D)\cap M}$. Fix two distinct points $x,y\in \overline{Q^{-1}(D)\cap M}$. Then (in the last inclusion we use Lemma \ref{lBasicPropertiesOfM})
$$Q(x),Q(y)\in Q(\overline{Q^{-1}(D)\cap M})\subset \overline{Q(Q^{-1}(D)\cap M)}\subset \overline{D\cap Q(M)}\subset \overline{D\cap M}.$$

We distinguish two cases. If $Q(x)\neq Q(y)$, then by the assumption and Theorem \ref{tRetractModel}, there exists a function $f\in\C(L)\cap M$ such that $f(Q(x))\neq f(Q(y))$. Using the elemetarity of $M$, $f\circ Q\in \C(K)\cap M$. Thus, the mapping $f\circ Q$ is the witness of the fact that $q_K^M(x)\neq q_K^M(y)$.

If $Q(x)= Q(y)$, then $x,y\in F$. By the elementarity of $M$, there is a countable set $S\in M$, $S\subset \C(K)$ such that $S$ separates the points of $F$. By Lemma \ref{lBasicPropertiesOfM}, $S\subset M$. Consequently, there exists a function $f\in C(K)\cap M$ such that $f(x)\neq f(y)$; hence, $q_K^M(x)\neq q_K^M(y)$.\end{proof}

\begin{lemma}\label{lQuotient2}Let $X$ be a Banach space and $Y$ its subspace such that $X/Y$ is separable. Let $i$ denote the injection of $Y$ into $X$ and $i^*$ its adjoint mapping. Let $K$ be a $w^*$-compact subset of $X^*$ and let $D\subset i^*(K)$ be a set induced by an $r$-skeleton in $i^*(K)$. If $(i^*)^{-1}(D)\cap K$ is dense in $K$, then the set $(i^*)^{-1}(D)\cap K$ is induced by an $r$-skeleton in $K$.
\end{lemma}
\begin{proof}Let us denote by $Q$ the canonical quotient mapping from $X$ onto $X/Y$. Then there is a countable set $S\subset X$ such that $Q(S)$ is dense in $X/Y$. By Theorem \ref{tRetractModel}, it is sufficient to prove that for every suitable elementary submodel $M$ such that $S, Y, K, X, i^* \in M$, the mapping $q_{K}^M$ is one-to-one on $\ov{(i^*)^{-1}(D)\cap M}\cap K$. Fix two distinct points $x^*,y^*\in \ov{(i^*)^{-1}(D)\cap M}\cap K$. Then (in the last inclusion we use Lemma \ref{lBasicPropertiesOfM})
$$i^*(x),i^*(y)\in i^*(\overline{(i^*)^{-1}(D)\cap M})\subset \overline{i^*((i^*)^{-1}(D)\cap M)}\subset \overline{D\cap i^*(M)}\subset \overline{D\cap M}.$$

We distinguish two cases. If $i^*(x^*)\neq i^*(y^*)$, then by the assumption and Theorem \ref{tRetractModel}, there exists a function $f\in\C(i^*(K))\cap M$ such that $f(i^*(x^*))\neq f(i^*(y^*))$. Using the elemetarity of $M$, $f\circ i^*\in \C(K)\cap M$. Thus, the mapping $f\circ i^*$ is the witness of the fact that $q_K^M(x^*)\neq q_K^M(y^*)$.

If $i^*(x^*) = i^*(y^*)$, then $0\neq x^*-y^*\in Y^\bot$. Using the fact that $Q(S)$ is dense in $X/Y$, there exists a point $z\in S\subset M$ such that $x^*-y^*(z)\neq 0$. Thus, the point $z\upharpoonright_{K}\in \C(K)\cap M$ is the witness of the fact that $q_K^M(x^*)\neq q_K^M(y^*)$.
\end{proof}

\section{Auxiliary results}

First, we give statements required in the prove of Theorem \ref{tMain1}. We begin with a lemma which is well known.

\begin{lemma}\label{lMetriz}Let $C\subset (X^*,w^*)$ be a countable compact. Then $\overline{\conv}^{w^*}C$ is metrizable.\end{lemma}
\begin{proof}As $C$ is countable compact, it is metrizable. Hence, $P(C)$ is metrizable. Now we observe (see {\cite[Lemma 4]{kalenda}}) that $\overline{\conv}^{w^*}C$ is a continuous image of a metrizable compact space $P(C)$; thus, it is metrizable as well (see {\cite[Theorem 4.4.15]{eng}}).\end{proof}

The following lemma and theorem were proved in the context of Valdivia compact spaces in {\cite[Proposition 3 and Theorem 1]{kalenda}}. In \cite{hajek} there are given proofs which work even for the setting of spaces from the class $\R_0$. Proofs contain some arguments which are not necessary, so we give simplified ones.

\begin{lemma}[cf. {\cite[Lemma 5.52]{hajek}}]\label{lNotRetract}Let $X$ be a Banach space such that $[0,\omega_1]$ embeds into $(B_{X^*},w^*)$. Let us have a point $e\in X$ and $\eps > 0$. Then there exists a $w^*$-compact and convex set $L\subset \{x^*\in X^*:\;x^*(e) = 0\}\cap \eps B_{X^*}$ that does not have a retractional skeleton.
\end{lemma}
\begin{proof}There exists $\{f_\alpha\}_{0\leq\alpha\leq\omega_1}\subset (B_{X^*},w^*)$ which is homeomorphic to $[0,\omega_1]$. We may without loss of generality assume that $f_{\omega_1} = 0$ and $\{f_\alpha\}_{0\leq\alpha\leq\omega_1}\subset (\eps B_{X^*},w^*)$. Moreover, fix a linearly independent set of points $\{e_k\}_{k=1}^\infty\subset X$ and functionals $\{g_k\}_{k=1}^\infty\subset \eps B_{X^*}$ such that $\|g_k\|\leq \tfrac{1}{k}$, $g_k(e)=0$ and $g_k(e_l)\neq 0$ if and only if $k=l$ (such a set of points and functionals exists - it is enough to create a biorthogonal system using the ``Gram-Schmidt orthogonalization process'', see {\cite[Lemma 1.21]{hajek}}). Since $f_\alpha(e)\to 0$ and, for all $k\in\en$, $f_\alpha(e_k)\to 0$ as $\alpha\to\omega_1$, there exists an $\alpha_0 < \omega_1$ such that $f_\alpha(e)= 0$ and $f_\alpha(e_k) = 0$  for all $k\in\en$, $\alpha\in[\alpha_0,\omega_1]$. Fix $\beta\in[\alpha_0,\omega_1)$. By Lemma \ref{lMetriz}, the set $L_\beta = \overline{\conv}^{w^*}(\{g_k\}_{k=1}^\infty\cup \{f_{\alpha}\}_{\alpha\in [\alpha_0,\beta)})$ is metrizable and thus there exists $\gamma_\beta\in(\beta,\omega_1)$ such that $f_\zeta\notin L_\beta$ whenever $\gamma_\beta\leq\zeta<\omega_1$ (otherwise some uncountable set $\{f_\zeta\}\subset L_\beta$ would contain a sequence converging to $f_{\omega_1}$, which is a contradiction). By the separation theorem, choose $y_\beta\in X$ such that $\sup y_\beta(L_\beta)<f_{\gamma_\beta}(y_\beta)$. Since $\lim_{\alpha\to\omega_1}f_\alpha(y_\beta) = 0$, $f_\alpha(y_\beta)=0$ for $\alpha$ large enough. Based on the above, we inductively find an increasing set $\{i(\beta)\}_{\beta < \omega_1}\subset [\alpha_0,\omega_1)$ such that $\sup_\beta i(\beta) = \omega_1$ and for every non limit ordinal $\beta < \omega_1$ there exists $y_\beta\in X$ satisfying $0\leq \sup y_{\beta}(L_{i(\beta)})<f_{i(\beta)}(y_\beta)$ and $f_{i(\gamma)}(y_\beta) = 0$ for all $\gamma > \beta$.
Let $L = \overline{\conv}^{w^*}(\{g_k\}_{k=1}^\infty\cup \{f_{i(\beta)}\}_{\beta < \omega_1})$. Then, for every non limit ordinal $\beta < \omega_1$, $f_{i(\beta)}$ is $w^*$-exposed by $y_\beta$ in $L$; hence, it is a $w^*$-$G_\delta$ point of $L$. Similarly, for all $k\in\en$, $g_k$ is $w^*$-exposed by $e_k$ in $L$ and all the functionals $g_k$ are $w^*$-$G_\delta$ points of $L$. By Corollary \ref{cNotSkeleton}, $L$ does not have a retractional skeleton.
\end{proof}

Now we are ready to prove the following theorem, which will be used in the proof of Theorem \ref{tMain1}.

\begin{thm}[cf. {\cite[Theorem 5.51]{hajek}}]\label{tNotSkeleton}Let $\langle X,\|\cdot\|\rangle$ be a Banach space such that $[0,\omega_1]$ embeds into $(B_{X^*},w^*)$. Then there is, for any $\eps\in(0,1)$, an equivalent norm $\tn\cdot \tn$ on $X$ such that $(1-\eps)\tn\cdot\tn\leq\|\cdot\|\leq\tn\cdot\tn$ and $(B_{\langle X^*,\tn\cdot \tn\rangle},w^*)\notin \R_0$.
\end{thm}
\begin{proof}Let us take an arbitrary $e\in S_X$ and $\eps\in(0,1)$. Then, by Lemma \ref{lNotRetract}, there exists a $w^*$-compact and convex set $L\subset \ker(e)\cap \eps B_{X^*}$ such that $L\notin\R_0$. Let us take an arbitrary $h\in S_{X^*}$ such that $h(e) = 1$. Then
$$B = \conv\{(L+h)\cup(-L-h)\cup(1-\eps)B_{X^*}\}$$
is a convex symmetric $w^*$-compact set such that $(1-\eps)B_{X^*}\subset B\subset (1+\eps)B_{X^*}$, so there is an equivalent norm $|\cdot|$ on $X$ such that $B$ is its dual unit ball. It remains to show that $B$ does not have a retractional skeleton (then we put $\tn\cdot\tn = (1+\eps)|\cdot|$ and this finishes the proof). Observe that
$$L+h = \{f\in B:\;f(e) = 1\}.$$
Thus, $L + h$ is a $w^*$-closed $w^*$-$G_\delta$ subset of $B$ and it does not have a retractional skeleton (because $L\notin\R_0$). By Theorem \ref{tDedicnost}, $B\notin\R_0$.
\end{proof}

Now we give some preliminary results which will be used in the proof of Theorem \ref{tMain3}. The following proposition is an analogue to {\cite[Proposition 1]{kalenda1}}.

\begin{proposition}\label{pGDeltaPointsNotCorson}Let $K$ be a compact space, $G$ the set of all $G_\delta$ points of $K$. If $\ov{G}$ is not Corson, then there are points $a,b\in K$ such that $P(L)\notin\R_0$ where $L$ is the quotient space made from $K$ by identifying $a$ and $b$.
\end{proposition}
\begin{proof}We use the same idea as in {\cite[Proposition 1]{kalenda1}}. If $\ov{G}\notin\R_0$, then we can take $a=b$ due to Lemma \ref{lInPKSpaces}. Now suppose that $\ov{G}$ has a retractional skeleton. Let $D$ be the unique set induced by an $r$-skeleton in $\ov{G}$ (the set is unique by Corollary \ref{cUnique}). Choose $a\in D$ a non-isolated point and $b\in \ov{G}\setminus D$ (such a point exists due to Theorem \ref{tCorson}). Let $L$ be the quotient space made from $K$ by identifying $a$ and $b$ and let $Q$ be the quotient mapping. Then $Q(\ov{G})$ does not have a retractional skeleton.

Indeed, in order to get a contradiction let $B\subset Q(\ov{G})$ be a set induced by an $r$-skeleton in $Q(\ov{G})$. Choose in the space $\ov{G}$ open neighborhoods $U$ and $V$ of $a$ and $b$ respectively with $\ov{U}\cap \ov{V} = \emptyset$. Then $U' = Q(U\setminus\{a\})$ and $V' = Q(U\setminus\{a\})$ are disjoint open sets with $\ov{U'}\cap\ov{V'} = \{\{a,b\}\}$. By Lemma \ref{lSubsetStability}, $\{a,b\}\in B$. Lemma \ref{lQuotient} shows that $(Q\upharpoonright_{\ov{G}})^{-1}(B)$ is induced by an $r$-skeleton in $\ov{G}$. By the uniqueness of $D$, $(Q\upharpoonright_{\ov{G}})^{-1}(B) = D$. This is a contradiction, because $b\in(Q\upharpoonright_{\ov{G}})^{-1}(B)\setminus D$.

Moreover, it is clear that $G\setminus\{a\}$ is dense in $\ov{G}$ and $Q(g)$ is a $G_\delta$ point in $L$ for every $g\in G\setminus\{a\}$. Thus, $Q(\ov{G})$ is the closure of all the $G_\delta$ points in $L$. By Lemma \ref{lInPKSpaces}, $P(L)\notin\R_0$.\end{proof}

To deal with compact spaces without $G_\delta$ points we use again the same approach as in \cite{kalenda1}.

\begin{proposition}\label{pContinuousImage}Let $K$ be a compact space such that there are two disjoint homeomorphic closed nowhere dense sets $M, N\subset K$ such that $N\notin\R_0$. Then there is $L$, an at most two-to-one continuous image of $K$, such that $L\notin \R_0$. Moreover, if $N$ has a dense set of (relatively) $G_\delta$ points, then $P(L)\notin\R_0$.
\end{proposition}
\begin{proof}We use the same idea as in {\cite[Proposition 2]{kalenda1}}. Let $h:M\to N$ be a homeomorphism and put $L = K\setminus M$ with the quotient topology defined by the mapping
$$\varphi(x) = \left\{
\begin{array}{ll}
	x & x\in K\setminus M\\
	h(x) & x\in M.
\end{array}\right.$$

There are disjoint open sets $U', V'$ in $K$ such that $U'\supset M$, $V'\supset N$ and $\overline{U'}\cap\overline{V'} = \emptyset$. Put $U = \varphi(U')\setminus N$ and $V = \varphi(V')\setminus N$. Then it follows from the definition of the quotient topology that $U$ and $V$ are disjoint open sets in $L$ and it is easy to see that $\overline{U}\cap\overline{V} = N$. Let us assume that $L\in\R_0$. Then, by Theorem \ref{tDedicnost}(ii), $N$ has a retractional skeleton, which is a contradiction.

Finally, let us assume that $N$ has a dense set of (relatively) $G_\delta$ points and $P(L)\in\R_0$. Copying word by word the arguments from {\cite[Proposition 2]{kalenda1}}, we observe that $P(\overline{W})$ is of the form $\bigcap_{n\in\en}\overline{G_n}$ with $G_n$ open in $P(L)$ whenever $W\subset L$ is open and that $P(N) = P(\overline{U})\cap P(\overline{V})$. By Theorem \ref{tDedicnost}(ii), $P(N)$ has a retractional skeleton. By Proposition \ref{pSkeletonAndPK}, $N$ has a retractional skeleton, which is a contradiction.
\end{proof}

The following Corollary is a generalization of {\cite[Corollary 1]{kalenda1}}. The prove can be done just by copying word by word the arguments from \cite{kalenda1}, using Example \ref{eNotRetractSkeleton} and Proposition \ref{pContinuousImage} instead of {\cite[Example 3.4]{kalenda2}} and {\cite[Proposition 2]{kalenda1}}.

\begin{cor}\label{cFourDisjointCopies}Let $K$ be a compact space which contains four pairwise disjoint nowhere dense homeomorphic copies of the ordinal segment $[0,\omega_1]$. Then there is $L$, at most four-to-one continuous image of $K$, such that $P(L)\notin\R_0$\end{cor}

In the proof of Theorem \ref{tMain2}, the following generalization of {\cite[Proposition 3]{kalenda1}} will be required.

\begin{proposition}\label{pCorsonWithoutM}Let $K$ be a Corson compact space without property $(M)$. Then there is a hyperplane $Y\subset \C(K)$ such that $B_{Y^*}\notin\R_0$.
\end{proposition}
\begin{proof}
 In {\cite[Proposition 3]{kalenda1}} it is observed that under the assumptions above, the following holds.

 $P(K)$ has a dense set of $G_\delta$ points, and a dense $\Sigma$-subset $A$. This set $A$ contains all the Dirac measures and there is a continuous measure $\mu\in P(K)\setminus A$. Take an arbitrary point $k$ from the support of the measure $\mu$ and put
 $$Y = \{f\in\C(K):\;f(k) = \mu(f)\}.$$
 Denote by $i$ the inclusion of $Y$ into $\C(K)$. Then $i^*(P(K))$ is a $w^*$-closed $w^*$-$G_\delta$ subset of $B_{Y^*}$.

 Now, in {\cite[Proposition 3]{kalenda1}} it is proved that $B_{Y^*}$ is not Valdivia. Let us see that it is not even in the class $\R_0$. For contradiction suppose $B_{Y^*}\in\R_0$.

 Then, by Theorem \ref{tDedicnost}, $i^*(P(K))\in\R_0$. Let $B$ be a set induced by an $r$-skeleton in $i^*(P(K))$. It follows from {\cite[Lemma 7]{kalenda1}} that $C = (i^*)^{-1}(B)\cap P(K)$ is dense in $P(K)$. By Lemma \ref{lQuotient2}, $C$ is induced by an $r$-skeleton in $P(K)$. As $P(K)$ has a dense set of $G_\delta$ points, $C=A$ by Corollary \ref{cUnique}.
 But $\delta_k\in A = C$ and also $\mu\notin A = C$. This is a contradiction with $i^*(\delta_k) = i^*(\mu)$.
\end{proof}

Finally, we observe that continuous images of spaces from the class $\R_0$ belong to the class $\GO$. The proof is again completely analogous to a similar result concerning Valdivia compacta {\cite[Proposition 4]{kalenda1}} (we only use Theorems \ref{tDedicnost} and \ref{tDichotomy} instead of {\cite[Lemma 5]{kalenda1}} and {\cite[Theorem 1]{kalenda3}}) and so it is omitted.

\begin{proposition}
 Lek $K$ be a compact space which is a continuous image of a space from the class $\R_0$. Then $K$ belongs to the class $\GO$.
\end{proposition}

\section{Proofs of the main results and open problems}

Without mentioning it any further, we will use two important results mentioned above. First, a Banach space is 1-Plichko if and only if it has a commutative 1-projectional skeleton. Next, a compact space is Valdivia if and only if it has a commutative retractional skeleton.

\begin{proof}[Proof of Theorem \ref{tMain1}]
 The implication (i)$\Rightarrow$(ii) comes from {\cite[Theorem 1]{kalenda}}. Obviously, (ii)$\Rightarrow$(iii) and (iii)$\Rightarrow$(iv). It follows from Proposition \ref{pSkeletonAndPK} that (ii)$\Rightarrow$(iii) holds. Finally, in order to prove (iii)$\Rightarrow$(i), let us assume that $B_{X^*}$ is not a Corson compact. If $B_{X^*}\notin\R_0$, we are done. If $B_{X^*}\in\R_0$, we use Theorems \ref{tDichotomy} and \ref{tNotSkeleton} to find an equivalent norm $\tn\cdot\tn$ such that $(B_{\langle X,\tn\cdot\tn\rangle},w^*)\notin\R_0$.
\end{proof}

\begin{proof}[Proof of Theorem \ref{tMain3}]
 It follows from Proposition \ref{pSkeletonAndPK} that (i)$\Rightarrow$(ii)$\Rightarrow$(iii)$\Rightarrow$(iv) hold. Finally, let $K$ be a non-Corson compact from the class $\GO$. Let $G$ be the set of $G_\delta$ points in $K$. If $\overline{G}$ is not Corson, we use Proposition \ref{pGDeltaPointsNotCorson} to get a two-to-one continuous image $L$ of $K$ such that $P(L)\notin\R_0$. If $\overline{G}$ is Corson, we copy word by word the arguments from the proof of (3)$\Rightarrow$(1) in {\cite[Theorem 2]{kalenda1}} to get a continuous image $L_0$ of $K$ such that it contains four pairwise disjoint nowhere dense homeomorphic copies of $[0,\omega_1]$. Now it is enough to use Corollary \ref{cFourDisjointCopies}.
\end{proof}

\begin{proof}[Proof of Theorem \ref{tMain2}]
 The implication (i)$\Rightarrow$(ii) comes from {\cite[Theorem 1]{kalenda1}}. It follows from Proposition \ref{pSkeletonAndPK} that (ii)$\Rightarrow$(iii) holds. Suppose that (iii) holds. By Theorem \ref{tMain3}, $K$ is Corson. If it had not the property $(M)$, we would get a contradiction with Proposition \ref{pCorsonWithoutM}.
\end{proof}

Theorem \ref{tMain} is just an immediate corollary of Theorem \ref{tMain3} and the well known fact that a continuous image of a Corson compact is again a Corson compact.\\

Finally, we state several open questions.

Given a compact space $K$ and a dense subset $D\subset K$, let $\tau_p(D)$ denote the topology of the pointwise convergence on $D$ (i.e. the weakest topology on $\C(K)$ such that $f\mapsto f(d)$ is continuous for every $d\in D$). Then $D$ is a $\Sigma$-subset of $K$ if and only if $D$ is countably closed and $(\C(K),\tau_p(D))$ is primarily Lindel\"of (see {\cite[Definition 1.2 and Theorem 2.1]{kalenda5}}).

\begin{problem}\label{problem1}Assume $D\subset K$ is a dense (resp. dense and countably closed) set in a compact space. Find a topological property $(T)$ of $(\C(K),\tau_p(D))$ such that $D$ is induced by an $r$-skeleton in $K$ if and only if $(\C(K),\tau_p(D))$ has the property $(T)$.
\end{problem}

For the motivation of the following question see Remark \ref{remark}.
\begin{question}\label{question1}Let $K$ be a compact space. Consider the following conditions
\begin{enumerate}[\upshape (i)]
	\item $\C(K)$ has a 1-projectional skeleton
	\item There is a convex symmetric set induced by an $r$-skeleton in $(B_{\C(K)^*},w^*)$
	\item There is a convex symmetric set induced by an $r$-skeleton in $P(K)$
\end{enumerate}Is it true that (iii)$\Rightarrow$(ii) (resp. (ii)$\Rightarrow$(i), resp. (iii)$\Rightarrow$(i))?
\end{question}

\begin{ack}
 The author would like to thank Ond\v{r}ej Kalenda for suggesting the topic and for many useful remarks and discussions.
\end{ack}


\begin{thebibliography}{10}
\bibitem{cuth} M.C\'uth: \emph{Separable reduction theorems by the method of elementary submodels}, Fund. Math., {\bf 219} (2012), 191--222.
\bibitem{DGZ}R. Deville, G. Godefroy, and V. Zizler, \emph{Smoothness and Renormings in Banach Spaces}, Pitman Monographs and Surveys in Pure and Applied Mathematics {\bf 64}, Longman Scientific and Technical, New York, 1993.
\bibitem{eng} R. Engelking: \emph{General Topology, Revised and completed edition}, Heldermann Verlag, Berlin (1989).
\bibitem{hajek} P.H\'ajek: V.Montesinos, J.Vanderwerff, V.Zizler: \emph{Biorthogonal Systems in Banach Spaces}, CMS Books, Springer, (2007).
\bibitem{kalenda1} O.Kalenda: \emph{Valdivia Compacta and Subspaces of $\C(K)$ Spaces}, Extracta Math. 14 (1999), no.3, 355-371.
\bibitem{kalenda2} O.Kalenda: \emph{Continuous Images and Other Topological Properties of Valdivia Compacta}, Fund. Math. 162 (1999), no.2, 181-192.
\bibitem{kalenda3} O.Kalenda: \emph{Embedding the Ordinal Segment $[0,\omega_1]$ into Continuous Images of Valdivia Compacta}, Comment. Math. Univ. Carolin. 40 (1999), no.4, 777-783.
\bibitem{kalenda} O.Kalenda: {\em Valdivia compacta and equivalent norms}, Studia Math. {\bf 138} (2000), 179-191.
\bibitem{kalenda5} O.Kalenda: {\em A characterization of Valdivia compact spaces}, Collect. Math. {\bf 51}, no. 1,(2000), 59-81
\bibitem{kalendaSurvey} O.Kalenda: \emph{Valdivia Compact Spaces in Topology and Banach Space Theory}, Extracta Math. 15 (2000), no.1, 1-85.
\bibitem{kalenda4} O.F.K.Kalenda: \emph{M-bases in spaces of continuous functions on ordinals}, Colloq. Math. {\bf 92}, no.2, (2002) 179–187.
\bibitem{kubisSmall} W.Kubi\'s, H.Michalewski: {\em Small Valdivia compact spaces}, Topology Appl. {\bf 153} (2006), 2560-2573.
\bibitem{kubisExample} T.Banakh, W.Kubi\'s: {\em Spaces of continuous functions over Dugundji compacta}, preprint, arXiv:math/0610795v2, (2008).
\bibitem{kubis} W.Kubi\'s: {\em Banach spaces with projectional skeletons}, J. Math. Anal. Appl. 350 (2009), no. 2, 758-776.
\bibitem{kubisKniha} J.Kakol, W.Kubi\'s, M.L\'opez-Pellicer: \emph{Descriptive Topology in Selected Topics of Functional Analysis Developments in Mathematics}, Vol. 24, Springer Science+Business Media, New York, (2011).
\bibitem{kunen} \label{kunen} K.Kunen: {\em Set Theory}, Stud. Logic Found. Math., vol. 102, North-Holland Publishing Co., Amsterdam, 1983.
\end{thebibliography}
\end{document}